\documentclass[a4paper,12pt]{article}
\usepackage{amssymb}
\usepackage{amsmath}
\usepackage{amsfonts}
\usepackage{amsthm}
\usepackage{amsmath}

\usepackage{tikz}
\usepackage{url}

\usepackage[english]{babel}
\usepackage[affil-it]{authblk}

  \makeatletter
\def\@fnsymbol#1{\ensuremath{\ifcase#1\or  \ddagger\or
		\mathsection\or \mathparagraph\or \|\or **\or \dagger\dagger
		\or \ddagger\ddagger \else\@ctrerr\fi}}
\makeatother

\newtheorem{theorem}{Theorem} 
\newtheorem{lemma}{Lemma}

\newtheorem{proposition}{Proposition}

\title{On extended perfect codes
}

\author{Konstantin Vorob'ev \\ Institute of Mathematics and Informatics, Bulgarian 
		Academy of Sciences,\\ Sofia 1113, Bulgaria \\konstantin.vorobev@gmail.com}

\begin{document}

\maketitle 

\begin{abstract}
	

We consider extended $1$-perfect codes in Hamming graphs $H(n,q)$. Such nontrivial codes are known only when $n=2^k$, $k\geq 1$, $q=2$, or  $n=q+2$, $q=2^m$, $m\geq 1$. Recently, Bespalov proved nonexistence of extended $1$-perfect codes for $q=3$, $4$, $n>q+2$. In this work, we characterize all positive integers $n$, $r$ and prime $p$, for which there exist such a code in $H(n,p^r)$. We also consider $2$-perfect codes in Hamming $H(n,q)$ and Johnson graphs $J(n,w)$ and find new necessary conditions on there existence.

\end{abstract}


\section{Introduction}\label{S:Introduction}

The {\emph{Hamming graph}} $H(n,q)$ is a graph whose vertices are the words of length $n$ over the alphabet  $\{0, 1, \dots, q-1\}$.
Two vertices are adjacent if and only if they differ in exactly one coordinate position, the Hamming distance $d(\bar x,\bar y)$ between vertices $\bar x$, $\bar y$ is the number of coordinates in which they differ. 
A subset $C$ of the set of vertices of $H(n,q)$ is called a {\emph{code}} in this graph.
A code distance $d$ of $C$ is the minimum distance between distinct elements of $C$.

A {\emph{Johnson graph}} $J(n,w)$, $n\geq 2w$, is a graph with the set of vertices ${[n] \choose w}$. Two vertices $x$, $y$ are adjacent if $|x \cap y|=w-1$.

A code $C$ in a simple graph $G$ is called a $r$-perfect code (perfect $r$-error-correction code), if the vertex set of $G$ is partitioned into balls of radius $r$ centered in the code vertices. It is known \cite{MWS} that if $q=p^m$ is a prime power, then there is a $1$-perfect code in $H(n, q)$ if and only if $n=\frac{q^t-1}{q-1}$ for some integer $t\geq 1$. In the case when $q$ is not a prime power, the problem of existence of $1$-perfect codes is far from being solved. Golomb and Posner proved \cite{GolombPosner64} nonexistence of $1$-perfect codes in $H(7,6)$ based on the nonexistence of a pair of orthogonal latin squares of
order $6$. Later Heden and Roos found \cite{HedRoos:2011} a necessary condition on existence of $1$-perfect codes, which in particular implies the non-existence of $1$-perfect codes in $H(19, 6)$. 
Golay \cite{Golay:49} also found two sporadic perfect codes with parameters $n=23$, $q=2$, $r=3$, and $n=11$, $q=3$ and $r=2$. For arbitrary $q$, it is known \cite{21e3ccf16d844538bb4cfae128246b7e, Best83, Hong1984OnTN} there are no nontrivial $r$-perfect codes in $H(n,q)$ for $q\geq 3$ and $r\geq 3$. In this work we are mainly interested in the case $r=2$. Reuvers showed \cite{21e3ccf16d844538bb4cfae128246b7e} the non-existence of such codes for $q=6,15,21,22,26,30,35$, van Lint  – for $q=10$ \cite{fc0961db19134eafb1bcbee8327dbde8}, and Bassalygo et al. \cite{BasZinLeo75} when $q=2^r3^s$. 
 
Another interesting class of codes in $H(n,q)$ is a class of so called extended $1$-perfect codes. A code $C$ in $H(n,q)$ is an extended $1$-perfect code if a puncturing (also known as a projection) of any coordinate position gives a $1$-perfect code in $H(n-1,q)$. It is known that such codes exist in $H(2^m, 2)$ and in $H(2^m + 2, 2^m)$ for any positive integer m (see \cite{MWS,BRZ:CR}). Formally, a one vertex in $H(2,q)$ is also a $1$-perfect code (we will call it trivial). Ball proved \cite{Ball:2012:1} that there are no linear extended $1$-prefect codes in $H(q + 2, q)$, when $q$ is odd prime. It is also known \cite{KKO:smallMDS} that there are no extended $1$-perfect codes in $H(7, 5)$ and $H(9, 7)$. Recently, Bespalov \cite{BESPALOV2022105607} proved nonexistence of extended $1$-perfect codes for $q=3$, $4$, $n>q+2$. 

In 1973, Delsarte formulated \cite{Delsarte:1973} a famous conjecture, that there are no nontrivial perfect codes in Johnson graphs. Here, by trivial one means a single vertex in $J(n,w)$ or two vertices at maximum distance in $J(2w,w)$ for odd $w$. 
For a Johnson graph, there is also an analogue of a Lloyd theorem and a method of $k$-regularity proposed by Etzion and Schwarz \cite{EtzionSchwartz2004}. Based on this method they showed that there are no $3$-perfect, $7$-perfect and $8$-perfect codes in Johnson graphs. There are also other necessary conditions on $(n,w,r)$ and computational results for particular values of $r$ of nonexistence up to some values of $n$. However, in general the problem of classification of such codes is far from being solved.

In this paper we are focused on extended $1$-perfect codes in Hamming graphs $H(n,q)$. Using a technique of weight distribution and some number-theoretical analysis we prove that when $q$ is a prime power, nontrivial extended $1$-perfect codes in Hamming graphs $H(n,q)$ exist only for $n$, $q$ listed above, i.e. in $H(2^m, 2)$ and $H(2^m + 2, 2^m)$, for positive integer $m$. Based on this technique we also find new necessary conditions on existence of $2$-perfect codes in Hamming $H(n,q)$ and Johnson graphs $J(n,w)$.


\section{Preliminaries}\label{S:Preliminaries}


An $r$-partition $(C_1,C_2,\dots, C_r)$ of the vertex set of a graph is called {\it equitable} with a {\it quotient matrix} $S=(s_{ij})_{i,j\in \{1,2,\dots, r\}}$ if every vertex from $C_i$ has exactly $s_{ij}$ neighbours in $C_j$. The sets $C_1, C_2,\dots, C_r$ are called {\it cells} of the partition.\\
Clearly, a $1$-perfect code in$ H(n,q)$ ia an equitable partition with the quotient matrix $\left( {\begin{array}{cc} 0 & n(q-1)  \\ 1 & n(q-1)-1   \end{array}} \right).$

Given a code $C$ in a simple graph $G = (V,E)$, the distance partition with respect to C is the partition
$\Pi=(C=C(0),C(1), \dots ,C(\rho))$ of $V$ with respect to the distance from code $C$, i.e. $C(i) = \{x \in V : d(x,C) = i\}$, and all $C(i)$ are nonempty.
The number $\rho = \rho(C)$ is defined by the equality $\bigcup\limits_{i=0}^{\rho}C(i)=V$ and is called the
covering radius of $C$. A nonempty code $C$ of a graph is called completely regular if the partition of the vertex set with respect to the distance from $C$ is equitable. 

For our arguments we need the following auxiliary statement.

\begin{proposition}\cite[Theorem 2]{BESPALOV2022105607}\label{prop:code_partition}.
	Given a code $C$ in $H(n,q)$, $C$ is an extended $1$-perfect code if and only if $(C=C(0), C(1), C(2))$ is an equitable partition in $H(n,q)$ with quotient matrix  $$S=\left( {\begin{array}{ccc} 0 & n(q-1) & 0 \\ 1 & q-2 & (n-1)(q-1) \\  0 & n & n(q-2) \end{array}} \right).$$
\end{proposition}

Let us define a matrix $S'=\frac{n(q-1)E-S}{q}=\frac{1}{q}\left( {\begin{array}{ccc} n(q-1) & -n(q-1) & 0 \\ -1 & (n-1)(q-1)+1 & -(n-1)(q-1) \\  0 & -n & n \end{array}} \right).$



Our next step is to find the Jordan canonical form of $S'$.

\begin{lemma}\label{l:jordan} 
	For $q\geq 2$ and $n\geq 2$, the equation $S'=QJQ^{-1}$ holds, where 	
	$$Q=\left( {\begin{array}{ccc} 1 & 1 & 1 \\ \frac{q - 2}{n(q-1)} & -\frac{1}{q -1} &   1 \\  -\frac{1}{(n-1)(q-1)} & \frac{1}{(q-1)^2} &  1 \end{array}} \right),
	J=\left( {\begin{array}{ccc}  \frac{nq-n-q+2}{q}& 0 & 0 \\ 0 & n & 0 \\ 0 & 0 & 0 \end{array}} \right),$$ 
	$$Q^{-1}=\left( {\begin{array}{ccc}  \frac{n(n - 1)(q - 1)}{(nq - n - q + 2)(n + q - 2)}& \frac{n(n - 1)(q - 1)(q - 2)}{(nq - n - q + 2)(n + q - 2)}  & -\frac{n(n - 1)(q - 1)^2}{(nq - n - q + 2)(n + q - 2)} \\ \frac{(q - 1)^2}{q(n + q - 2)} & -\frac{n(q - 1)^2}{q(n + q - 2)} & \frac{(n - 1)(q - 1)^2}{q(n + q - 2)} \\  \frac{1}{q(nq - n - q + 2)} & \frac{n(q-1)}{q(nq - n - q + 2)} & \frac{(n-1)(q-1)^2}{q(nq - n - q + 2)}  \end{array}} \right).$$
	
\end{lemma}
\begin{proof}
	The proof is obtained by direct calculations. Let us also notice that $$\det(Q)=-\frac{q(nq - n - q + 2)(n + q - 2)}{n(n - 1)(q - 1)^3}\neq 0,$$ for $n\geq2$ and $q\geq 2$.

\end{proof}
Evidently, the equation from Lemma \ref{l:jordan} corresponds to a reduction to the Jordan canonical form of the matrix $S'$. Hence, diagonal elements of $J$ are the eigenvalues of $S'$. For our following arguments we also need to know values of Krawtchouk polynomials on these eigenvalues.

Well-known family of orthogonal Krawtchouk polynomials is defined by the following equations
$$K_r(x,q,n)=\sum_{j=0}^{r}{(-1)^{j}q^{r-j}{n-r+j \choose j}{n-x \choose r-j}},$$
where $n$, $q$ and degree $r$ are nonnegative integers.

\begin{lemma}\label{l:Krawtchouk} 
  For $q,n,x\in \mathbb{Z}$, $q\geq 2$, $n\geq 2$, $0 \leq x \leq n$,we have that $$K_n(x,q,n)=(-1)^{x}(q-1)^{n-x}.$$ 
 
\end{lemma}

\begin{proof}
By definition of Krawtchouk polynomial for $r=n$, it is easy to compute  $$K_n(x,q,n)=\sum\limits_{j=0}^{n}{(-1)^jq^{n-j}\binom{n-x}{n-j}}=(-1)^x\sum\limits_{j'=0}^{n-x}(-1)^{j'}q^{n-x-j'}\binom{n-x}{n-x-j'}=(-1)^{x}(q-1)^{n-x}.$$

\end{proof}

Now let us turn back to extended $1$-perfect codes. We know that its distance partition $\Pi$ is an equitable partition.  Takes some integer $i$, $2\leq i \leq n$ and consider a graph $H_i(n,q)$ with the same vertex set and two vertices being adjacent if the distance (here, Hamming distance) between them equals $i$. The natural question here, is $\Pi$ equitable also in this new graph? The answer is positive, and not only for the Hamming graph. Actually, any equitable partition of a given distance regular graph $G$ will be equitable for distance-$i$ graph $G$, see \cite[Sect. 2.2.2, 2.1.5]{Martin:PhD} and \cite{Kro:struct}. Moreover, there is a way to find the corresponding quotient matrix of the partition in a new graph directly. In the following Lemma we make it for $H_n(n,q)$.    

\begin{lemma}\label{l:lastweight} 
	Given an extended $1$-perfect code $C$, then the distance partition $\Pi$ of $C$ in $H(n,q)$ is equitable also in $H_n(n,q)$. Corresponding quotient matrix in the last graph equals $QMQ^{-1}$, where
	 $$M=\left( {\begin{array}{ccc} (-1)^{\frac{nq-n-q+2}{q}}(q-1)^{\frac{n+q-2}{q}} & 0 & 0 \\ 0 & (-1)^n & 0 \\  0 & 0 & (q-1)^n \end{array}} \right),$$
	 and matrices $Q$ and $Q^{-1}$ are defined in Lemma \ref{l:jordan}.  
\end{lemma}

\begin{proof}
	As we discussed before, $\Pi$ is equitable in $H_n(n,q)$. By \cite[Theorem 3 and Example 8]{Kro:struct}
	the corresponding quotient matrix equals $K_n(S',q,n)=QK_n(J,q,n)Q^{-1}$. Since $J$ is a diagonal matrix we have that $$K_n(J,q,n)=\left( {\begin{array}{ccc} K_n(nq-n-q+2,q,n) & 0 & 0 \\ 0 &  K_n(n,q,n) & 0 \\  0 & 0 &  K_n(0,q,n) \end{array}} \right).$$
	The last step it to apply Lemma \ref{l:Krawtchouk} and we finish. 
\end{proof}

\section{Extended perfect codes}\label{S:Extended}
From Section \ref{S:Introduction} we remember, that extended $1$-perfect codes are known only when $n$ and $q$ belong to the following three infinite families. 
\begin{proposition}\label{prop:known_codes}
	Extended $1$-perfect codes in $H(n,q)$ exist for the following values of $n$ and $q$:
	\begin{enumerate}
		\item $n=2$, $q\geq 2$ (one vertex in $H(2,q)$),
		\item $n=q+2$, $q=2^m$, $m\geq 1$ (extended perfect code in $H(q+2,q)$).
		\item $n=2^k$, $k\geq 1$, $q=2$ (extended perfect code in $H(2^k,2)$).
	\end{enumerate}
\end{proposition}


 The main idea of this Section is to find the quotient matrix of the initial distance partition in $H_n(n,q)$ in closed form.  

\begin{theorem}\label{T:general_q}
	Let $C$ be an extended $1$-perfect code in $H(n,q)$, $q\geq 2$, $n\geq 2$. Then $$(q-2)\frac{(q-1)^x-(-1)^x}{x} \in \mathbb{N},$$ where $x=\frac{n+q-2}{q}$. 	
\end{theorem}
\begin{proof}
Consider a distance partition $\Pi=(C=C(0), C(1), C(2))$ defined by $C$, with the quotient matrix $$S=\left( {\begin{array}{ccc} 0 & n(q-1) & 0 \\ 1 & q-2 & (n-1)(q-1) \\  0 & n & n(q-2) \end{array}} \right).$$ 
By Lemma \ref{l:lastweight} we have that the quotient matrix of $\Pi$ in $H_n(n,q)$ equals  $M'=QMQ^{-1}$, where
$$M'=\left( {\begin{array}{ccc} (-1)^{\frac{nq-n-q+2}{q}}(q-1)^{\frac{n+q-2}{q}} & 0 & 0 \\ 0 & (-1)^n & 0 \\  0 & 0 & (q-1)^n \end{array}} \right),$$  
and matrices $Q$ and $Q^{-1}$ were defined in Lemma \ref{l:jordan}.  
Let us count the entry of $M$ in the third row and the first column. After all necessary calculations we have that

\begin{multline*}
	(M')_{3,1}=\frac{1}{q(nq-n-q+2)}\bigg((q-1)^n-(-1)^n+q\big((-1)^n-(q-1)^{\frac{q+n-2}{q}}(-1)^{\frac{nq-n-q}{q}}\big)+ \\ +(q-2)\frac{(q-1)^{\frac{n+q-2}{q}}(-1)^{\frac{nq-q-n+2}{q}} - (-1)^n}{\frac{n+q-2}{q}}\bigg).
\end{multline*}	
	
By definition of weight distribution this value is nonnegative integer. Clearly, the last summand in big brackets must be integer. Therefore,
$$(q-2)\frac{(q-1)^{\frac{n+q-2}{q}}(-1)^{\frac{nq-q-n+2}{q}} - (-1)^n}{\frac{n+q-2}{q}} \in \mathbb{Z}.$$
After multiplying it by $(-1)^{\frac{nq-q-n+2}{q}}$ we have the required.
\end{proof}
	
This new necessary condition proved in Theorem \ref{T:general_q} works for arbitrary $q$. However, we are mainly interested in the case $q=p^m$, where $p$ is a prime number and $m\geq 1$. The next step is to classify all such pairs of $n$ and $q$ that there exist extended $1$-perfect code in $H(n,q)$ and $q$ is a prime power.

\begin{theorem}\label{T:prime}
	Let $C$ be an extended $1$-perfect code in $H(n,q)$ for $q=p^m$, where $p$ is a prime number $m\geq 1$, $n\geq 2$. Then one of the following holds:
	\begin{enumerate}
		\item $n=2$,
		\item $p=2$, $m=1$ and $n=2^t$ for some integer $t\geq 2$,
		\item $p=2$, $m\geq 1$, $n=q+2$.
	\end{enumerate}
	
\end{theorem}

\begin{proof}
	We consider the case when $q$ is a prime power. Since $1$-perfect codes exist in $H(n',q)$ only for $n'=\frac{q^k-1}{q-1}$ for some integer $k\geq 1$ we conclude that
	
	\begin{equation}\label{e:0}
		n=\frac{q^k-1}{q-1}+1,
	\end{equation} 
for some integer $k\geq 1$. 
	
	The case $k=1$ immediately lead us a to the one-vertex code in $H(n,2)$. In the rest of the proof we suppose that $k\geq 2$.
	Our next step is to use Theorem \ref{T:general_q}. Indeed, for $n=\frac{q^k-1}{q-1}+1$ we have that
	\begin{equation}\label{e:1}
		r=(q-2)\frac{(q-1)^x-(-1)^x}{x} \in \mathbb{N},
	\end{equation}
 where 
 	\begin{equation}\label{e:2}
 		x=\frac{q^{k-1}+q-2}{q-1} \in \mathbb{N},
 	\end{equation}	
Let us rewrite this equation in the following form
 	\begin{equation}\label{e:3}
	x(q-1)=q^{k-1}+q-2 
\end{equation}	
	
	The following analysis depends on the parity of $p$.
	
	\textbf{Case $p>2$.} Evidently, it follows from (\ref{e:1}) that $x$ is odd too. Since $k\geq 2$, from (\ref{e:2}) we have $x>1$.  We know that $q$ is odd, so $(q,q-2)=1$ (here and further by $(a,b)$ we mean the greatest common divisor of two positive integers $a$ and $b$.). Simple analysis of (\ref{e:1}) and (\ref{e:3}) shows that
	 	\begin{equation}\label{e:4}
		(x,q)=(x,q-1)=(x,q-2)=1. 
	\end{equation}	
	
	Using this fact and $(\ref{e:1})$ we have that  $\frac{(q-1)^x-(-1)^x}{x} \in \mathbb{N}$.
	Let $t$ be the smallest odd prime divisor of $x$. Since $x>1$ and $x$ is odd, $t$ is not less than $3$. 
	Since $(q-1)^x+1$ is divisible by $t$, we conclude that $(q-1)^{2x}-1$ is divisible by $t$. Equality (\ref{e:4}) implies $(t,q-1)=1$. Consequently, $(q-1)^{t-1}-1$ is also divisible by $t$. These two divisions guarantee us that there exists the smallest natural number $d$ such that  $(q-1)^{d}-1$ is divisible by $t$, and $d$ divides $(t-1,2x)$. It follows from the choice of $t$ that $d$ does not have odd prime divisors. 
	
	Combining this fact with the one that $2x$ is not divisible by $4$, we have that $d=1$ or $d=2$. In other words, $q$ or $q(q-2)$ is divisible by $t$ correspondingly. These cases both contradict to (\ref{e:4}).     
	
	\textbf{Case $p=2$.} First of all, if $m=1$ (i.e. $q=2$) then extended $1$-perfect codes exist in $H(2^k,2)$ for all $k\geq 2$. Hence, further we consider the case $m\geq 2$. If $k=2$ then we have $n=q+2$ where $q$ is a power of $2$. These are exactly parameters of the codes from another known family of codes. To finish the prove we need to consider the case $k\geq 3$. The idea is to repeat the arguments from the previous case for $\frac{x}{2}$.  Clearly, it follows from (\ref{e:2}) that $x$ is even but not divisible by $4$. Hence, $\frac{x}{2}$ is odd. Since $k\geq 3$, we have $\frac{x}{2}>1$. Again, analysis of (\ref{e:1}) and (\ref{e:3}) shows that
	\begin{equation}\label{e:5}
		(\frac{x}{2},q)=(\frac{x}{2},q-1)=(\frac{x}{2},q-2)=1. 
	\end{equation}	
	Define $t$ as the smallest odd prime divisor of $\frac{x}{2}$. As before, $t\geq 3$. Now we have that  $(q-1)^{x}-1$ and $(q-1)^t-1$ are both divisible by $t$. Hence, there exist the smallest natural $d$ such that $d$ divides $(t-1,x)$. By choice of $t$ and the fact that $x$ is not divisible by $4$, we have that $d=1$ or $d=2$. Consequently,  $q$ or $q(q-2)$ is divisible by $t$ correspondingly, which contradicts to (\ref{e:5}). 
\end{proof}

 The proof is relatively short and is based on the fact that $q$ ia prime power. Unfortunately, the idea of the proof can not be applied to the case of arbitrary $q$ directly. Indeed, in general case one has more complicated condition instead of equation (\ref{e:0}). Moreover, as we discussed in Section \ref{S:Introduction}, one does not know any nontrivial $1$-perfect codes when $q$ is not a prime number. So the generalization of ideas from the proof of Theorem \ref{T:prime} might make sense only when one find some series of $1$-perfect codes in this case.

 \section{$2$-Perfect codes in $H(n,q)$}\label{S:2-perfect_H}
 
 \begin{proposition}\label{prop:perf_code_partition}
 	Given a code $C$ in $H(n,q)$, $C$ is a $2$-perfect code if and only if $(C=C(0), C(1), C(2))$ is an equitable partition in $H(n,q)$ with quotient matrix  $$S=\left( {\begin{array}{ccc} 0 & n(q-1) & 0 \\ 1 & q-2 & (n-1)(q-1) \\  0 & 2 & n(q-1)-2 \end{array}} \right).$$
 \end{proposition}
 \begin{proof} The statement follows from the definition of a $2$-perfect code.
 	\end{proof} 
 Let us define a matrix $S'=\frac{n(q-1)E-S}{q}=\frac{1}{q}\left( {\begin{array}{ccc} n(q-1) & -n(q-1) & 0 \\ -1 & (n-1)(q-1)+1 & -(n-1)(q-1) \\  0 & -2 & 2 \end{array}} \right).$
 Matrix $S$ has eigenvalues $n(q-1)$ and $\frac{q}{2}\pm \frac{\sqrt{4(n-2)q+q^2-4n+8}}{2}-2$. By Lloyd theorem these eigenvalues must be also eigenvalues of $H(n,q)$. Direct calculation show that fact is equivalent to the following parameterization (that will be used until the end of this Section):
 $$n=\frac{t^2-q^2+8q-8}{4(q-1)},\, t=qm, $$
 where $m$ is a positive integer.
 In these terms,
 $$S'=\frac{n(q-1)E-S}{q}=\frac{1}{q}\left( {\begin{array}{ccc} \frac{(m^2 - 1)q^2 + 8q - 8}{4} & -\frac{(m^2 - 1)q^2 + 8q - 8}{4} & 0 \\ -1 & \frac{(m^2 - 1)q^2 + 4q}{4} & -\frac{(m^2 - 1)q^2 + 4q - 4}{4} \\  0 & -2 & 2 \end{array}} \right).$$

 
 
 Our next step is to find the Jordan canonical form of $S'$.
 
 \begin{lemma}\label{l:jordan2} 
 	For $q\geq 2$ and $n\geq 2$, the equation $S'=QJQ^{-1}$ holds, where 	
 	$$Q=\left( {\begin{array}{ccc} 1 & 1 & 1 \\ \frac{2((m + 1)q - 4)}{(m^2 - 1)q^2 + 8q - 8}   &  -\frac{2((m - 1)q + 4)}{(m^2 - 1)q^2 + 8q - 8} &   1 \\  -\frac{16}{(m^3 - m^2 - m + 1)q^3 + 2(m^2 + 4m - 5)q^2 - 8(m - 3)q - 16} & \frac{16}{(m^3 + m^2 - m - 1)q^3 - 2(m^2 - 4m - 5)q^2 - 8(m + 3)q + 16}  &  1 \end{array}} \right),$$
 	
 	$$J=\left( {\begin{array}{ccc} s_1 & 0 & 0 \\ 0 & s_2 & 0 \\ 0 & 0 & 0 \end{array}} \right),\, Q^{-1}=$$ 
 	
 	$$=\left( {\begin{array}{ccc}   \frac{(m^2q^2 - q^2 + 8q - 8)(mq - q + 2)}{2(m^2q - 2m - q + 6)mq^2}   &  \frac{(m^2q^2 - q^2 + 8q - 8)(mq + q - 4)(mq - q + 2)}{4(m^2q - 2m - q + 6)mq^2}   & -\frac{(m^2q^2 - q^2 + 8q - 8)(mq + q - 2)(mq - q + 2)}{4(m^2q - 2m - q + 6)mq^2} \\  \frac{(m^2q^2 - q^2 + 8q - 8)(mq + q - 2)}{2(m^2q + 2m - q + 6)mq^2}  & -\frac{(m^2q^2 - q^2 + 8q - 8)(mq + q - 2)(mq - q + 4)}{4(m^2q + 2m - q + 6)mq^2} &  \frac{(m^2q^2 - q^2 + 8q - 8)(mq + q - 2)(mq - q + 2)}{4(m^2q + 2m - q + 6)mq^2} \\  \frac{32}{(m^2q + 2m - q + 6)(m^2q - 2m - q + 6)q^2}  &  \frac{8(m^2q^2 - q^2 + 8q - 8)}{(m^2q + 2m - q + 6)(m^2q - 2m - q + 6)q^2}   & \frac{(m^2q^2 - q^2 + 8q - 8)(mq + q - 2)(mq - q + 2)}{(m^2q + 2m - q + 6)(m^2q - 2m - q + 6)q^2} \end{array}} \right),$$
 	where $s_1=\frac{m^2q}{4} - \frac{m}{2} - \frac{q}{4} + \frac{3}{2}$ and $s_2=\frac{m^2q}{4} +\frac{m}{2} - \frac{q}{4} + \frac{3}{2}$.

 \end{lemma}
 \begin{proof}
 	The proof is obtained by direct calculations.
 	
 \end{proof}

Here we repeat arguments from Section \ref{S:Preliminaries} and prove the following lemma.

\begin{lemma}\label{l:lastweight_2} 
	Given a $2$-perfect code $C$, then the distance partition $\Pi$ of $C$ in $H(n,q)$ is equitable also in $H_n(n,q)$. Corresponding quotient matrix in the last graph equals $QMQ^{-1}$, where
	$$M=\left( {\begin{array}{ccc} (-1)^{s_1}(q-1)^{n-s_1} & 0 & 0 \\ 0 & (-1)^{s_2}(q-1)^{n-s_2} & 0 \\  0 & 0 & (q-1)^n \end{array}} \right),$$
	and matrices $Q$ and $Q^{-1}$ are defined in Lemma \ref{l:jordan2}.  
\end{lemma}

\begin{proof}
	As we discussed before, $\Pi$ is equitable in $H_n(n,q)$. By \cite[Theorem 3 and Example 8]{Kro:struct}
	the corresponding quotient matrix equals $K_n(S',q,n)=QK_n(J,q,n)Q^{-1}$. Since $J$ is a diagonal matrix we have that $$K_n(J,q,n)=\left( {\begin{array}{ccc} K_n(nq-n-q+2,q,n) & 0 & 0 \\ 0 &  K_n(n,q,n) & 0 \\  0 & 0 &  K_n(0,q,n) \end{array}} \right).$$
	The last step it to apply Lemma \ref{l:Krawtchouk} and we finish. 
\end{proof}

 \begin{theorem}\label{T:general_q_2}
 	Let $C$ be a $2$-perfect code in $H(n,q)$, $q\geq 2$, $n=\frac{t^2-q^2+8q-8}{4(q-1)}$, $t=qm$, $m\geq 1$. Then $$2m(q-1)^{n-s_1}\frac{(q-1)^{s_1}+(-1)^{s_1}}{s_1} \in \mathbb{N}.$$ 	
 \end{theorem}
 \begin{proof}
 	Consider a distance partition $\Pi=(C=C(0), C(1), C(2))$ defined by $C$, with the quotient matrix $$S=\left( {\begin{array}{ccc} 0 & n(q-1) & 0 \\ 1 & q-2 & (n-1)(q-1) \\  0 & 2 & n(q-1)-2 \end{array}} \right).$$ 
 	By Lemma \ref{l:lastweight_2} we have that the quotient matrix of $\Pi$ in $H_n(n,q)$ equals  $M'=QMQ^{-1}$, where
 	$$M=\left( {\begin{array}{ccc} (-1)^{s_2}(q-1)^{n-s_2} & 0 & 0 \\ 0 & (-1)^{s_1}(q-1)^{n-s_1} & 0 \\  0 & 0 & (q-1)^n \end{array}} \right),$$
 	  
 	and matrices $Q$ and $Q^{-1}$ were defined in Lemma \ref{l:jordan2}.  
 	Let us count the entry of $M'$ in the third row and the first column. After all necessary calculations we have that

 	\begin{multline*}
 		(M')_{3,1}=\frac{1}{q^2ms_2}\bigg(\frac{2s_2(-1)^{s_1}(q-1)^{n-s_1}+2m(q-1)^n}{s_1}-2(-1)^{s_2}(q-1)^{n-s_2}  \bigg).
 	\end{multline*}	
 	
 	By definition of weight distribution this value is nonnegative integer. Clearly, the first summand in big brackets must be integer. Therefore,
 	$$\frac{2s_2(-1)^{s_1}(q-1)^{n-s_1}+2m(q-1)^n}{s_1}\in \mathbb{Z}.$$ Since $s_2=m+s_1$, we finish the proof.
 	
 \end{proof}

 \section{$2$-Perfect codes in $J(n,w)$}\label{S:2-perfect_J}

 \begin{proposition}\label{prop:code_partition_J}.
 	Given a code $C$ in $J(n,w)$, $C$ is a $2$-perfect code if and only if $(C=C(0), C(1), C(2))$ is an equitable partition in $J(n,w)$ with quotient matrix  $$S=\left( {\begin{array}{ccc} 0 & w(n-w) & 0 \\ 1 & n-2 & w(n-w)-n+1 \\  0 & 4 & w(n-w)-4 \end{array}} \right).$$
 \end{proposition}
  \begin{proof} The statement follows from the definition of a $2$-perfect code.
 \end{proof} 
Let us notice, that matrix $S$ has eigenvalues $w(n-w)$ and $\frac{n \pm t-6}{2}$, where $t=\sqrt{n^2+4wn-4w^2-12n+20}$. These eigenvalues also must be eigenvalues of $J(n,w)$.
 
  Let us define a matrix $S'=\frac{(n+1)E-\sqrt{T}}{2}$ where $T=(n+1)^2E-4(w(n-w)E-S)$.  
 
 
 Our next step is to find the Jordan canonical forms of $T$ and $S'$.
 
 \begin{lemma}\label{l:jordan_J} 
 	For $w\geq 2$ and $n\geq 2w$, the equation $T=QJQ^{-1}$ holds, where 	
 	$$Q=\frac{1}{8}\left( {\begin{array}{ccc} 8 & tw(n-w)+n^2w-nw^2-4nw+4w^2 & -tw(n-w)+n^2w-nw^2-4nw+4w^2 \\ 8 & -t-2nw+n+2w^2+2 &   t-2nw+n+2w^2+2 \\  8 & 8 &  8 \end{array}} \right),$$
 	
 	$$J=\left( {\begin{array}{ccc}  (n+1)^2& 0 & 0 \\ 0 & -2t+n^2-4n(w-1)+4w^2-11 & 0 \\ 0 & 0 & 2t+n^2-4n(w-1)+4w^2-11 \end{array}} \right),$$ 
 	
 	$$Q^{-1}=\frac{1}{(n^2(w-1)w+n(-2w^2+w+5)w+w^4-5w^2+4)}\cdot $$
 	$$\left( {\begin{array}{ccc} 4 & 4w(n-w) & n^2w^2-n^2w-2nw^3+nw^2+nw+w^4-w^2 \\ -\frac{2(t-2nw+n+2w^2-6)}{t} & -\frac{2(tw(n-w)-n^2w+nw^2+4nw-4w^2+8)}{t} & \frac{2(t(nw-w^2+1)-n^2w+nw^2+2nw+n-2w^2+2)}{t}\\ -\frac{2(t+2nw-n-2w^2+6)}{t}  & \frac{2(-tw(n-w)-n^2w+nw^2+4nw-4w^2+8)}{t}  & \frac{2(t(nw-w^2+1)+n^2w-nw^2-2nw-n+2w^2-2)}{t} \end{array}} \right).$$

 \end{lemma}
 \begin{proof}
 	The proof is obtained by direct calculations. 
 	
 \end{proof}

By definition of $S'$ we immediately have the following lemma.
\begin{lemma}\label{l:jordan_J_2} 
	For $w\geq 2$ and $n\geq 2w$, the equation $S'=QJ'Q^{-1}$ holds, where

	$$J'=\left( {\begin{array}{ccc}  0 & 0 & 0 \\ 0 & n+1-\sqrt{-2t+n^2-4n(w-1)+4w^2-11} & 0 \\ 0 & 0 & n+1-\sqrt{2t+n^2-4n(w-1)+4w^2-11} \end{array}} \right),$$ 
	and matrices $Q$, $Q^{-1}$ are from Lemma \ref{l:jordan_J}.

\end{lemma}

 Evidently, the equation from Lemma \ref{l:jordan_J_2} corresponds to a reduction to the Jordan canonical form of the matrix $S'$. Hence, diagonal elements of $J$ are the eigenvalues of $S'$. For our following arguments we also need to know values of Eberlein polynomials on these eigenvalues.
 
 Well-known family of orthogonal Eberlein polynomials is defined by the following equations
 $$E_r(x,w,n)=\sum_{j=0}^{r}{(-1)^{j}{x \choose j}{w-x \choose r-j}{n-w-x \choose r-j}},$$
 where $n$, $w$ and degree $r$ are nonnegative integers.
 
 \begin{lemma}\label{l:Eberlein} 
 	For $n,w,x\in \mathbb{Z}$, $w\geq 2$, $n\geq 2w$, $0 \leq x \leq w$,we have that $$E_w(x,w,n)=(-1)^x{n-w-x \choose n-2w}.$$ 
 	
 \end{lemma}
 
 \begin{proof}
 	By definition of Eberlein polynomial for $r=w$, it is easy to compute  $$E_w(x,w,n)=\sum_{j=0}^{w}{(-1)^{j}{x \choose j}{w-x \choose w-j}{n-w-x \choose w-j}}=(-1)^x{n-w-x \choose w-x}=(-1)^x{n-w-x \choose n-2w}.$$
 	
 \end{proof}

\begin{lemma}\label{l:lastweight_3} 
	Given a $2$-perfect code $C$, then the distance partition $\Pi$ of $C$ in $J(n,w)$ is equitable also in $J_w(n,w)$. Corresponding quotient matrix in the last graph equals $QMQ^{-1}$, where
	$$M=\left( {\begin{array}{ccc} {n-w \choose w} & 0 & 0 \\ 0 & (-1)^{n+1-\sqrt{-2t+h}}{\sqrt{-2t+h}-w-1 \choose n-2w} & 0 \\  0 & 0 & (-1)^{n+1-\sqrt{2t+h}}{\sqrt{2t+h}-w-1 \choose n-2w} \end{array}} \right),$$
	$h=n^2-4n(w-1)+4w^2-11$ and matrices $Q$ and $Q^{-1}$ are defined in Lemma \ref{l:jordan_J}.  
\end{lemma}

\begin{proof}
	As we discussed before, $\Pi$ is equitable in $J_w(w,n)$. By \cite[Theorem 3 and Example 9]{Kro:struct}
	the corresponding quotient matrix equals $E_w(S',w,n)=QE_w(J',w,n)Q^{-1}$. Since $J'$ is a diagonal matrix we have that $$E_w(J',w,n)=
	\left( {\begin{array}{ccc} E_w(0,w,n) & 0 & 0 \\ 0 &  E_w(n+1-\sqrt{-2t+h},w,n) & 0 \\  0 & 0 &  E_w(n+1-\sqrt{2t+h},w,n) \end{array}} \right).$$
	The last step it to apply Lemma \ref{l:Eberlein} and we finish. 
\end{proof}

\begin{theorem}\label{T:general_J}
	Let $C$ be a $2$-perfect code in $J(n,w)$, $w\geq 2$, $n\geq 2w$. Then $$\frac{2(2nw-n-2w^2+6)}{t}\bigg({\sqrt{-2t+h}-w-1 \choose n-2w}-{\sqrt{2t+h}-w-1 \choose n-2w}  \bigg) \in \mathbb{Z},$$ where $h=n^2-4n(w-1)+4w^2-11$ and $t=\sqrt{n^2+4wn-4w^2-12n+20}$. 	
\end{theorem}
\begin{proof}
	Consider a distance partition $\Pi=(C=C(0), C(1), C(2))$ defined by $C$, with the quotient matrix $$S=\left( {\begin{array}{ccc} 0 & w(n-w) & 0 \\ 1 & n-2 & w(n-w)-n+1 \\  0 & 4 & w(n-w)-4 \end{array}} \right).$$
	
	By Lemma \ref{l:lastweight_3} we have that the quotient matrix of $\Pi$ in $J_w(w,n)$ equals  $M'=QMQ^{-1}$, where
	$$M=\left( {\begin{array}{ccc} {n-w \choose w} & 0 & 0 \\ 0 & (-1)^{n+1-\sqrt{-2t+h}}{\sqrt{-2t+h}-w-1 \choose n-2w} & 0 \\  0 & 0 & (-1)^{n+1-\sqrt{2t+h}}{\sqrt{2t+h}-w-1 \choose n-2w} \end{array}} \right),$$
$h=n^2-4n(w-1)+4w^2-11$ and matrices $Q$ and $Q^{-1}$ are defined in Lemma \ref{l:jordan_J}.    
	Let us count the entry of $M'$ in the third row and the first column. After all necessary calculations we have that

	\begin{multline*}
		(M')_{3,1}=\frac{1}{(n^2(w-1)w+n(-2w^2+w+5)w+w^4-5w^2+4)}\cdot \\ \bigg( 4{n-w \choose w}- (-1)^{n+1-\sqrt{-2t+h}}{\sqrt{-2t+h}-w-1 \choose n-2w}\frac{2(t-2nw+n+2w^2-6)}{t}- \\- (-1)^{n+1-\sqrt{2t+h}}{\sqrt{2t+h}-w-1 \choose n-2w}\frac{2(t+2nw-n-2w^2+6)}{t}\bigg).
	\end{multline*}	
	
	By definition of weight distribution this value is nonnegative integer. Clearly, the sum of two last summands in big brackets must be integer. Therefore,
	$$\frac{2(2nw-n-2w^2+6)}{t}\bigg({\sqrt{-2t+h}-w-1 \choose n-2w}-{\sqrt{2t+h}-w-1 \choose n-2w}  \bigg) \in \mathbb{Z}.$$
	
\end{proof}

This new necessary condition proved in Theorem \ref{T:general_J} works for arbitrary $w$ and $n$.
 
\section{Conclusion}\label{S:Conclusion}

In this work we completely characterize all numbers $n$, $p$ and $m$, for which there exists an extended $1$-perfect code in $H(n,q)$, $q=p^m$. The situation for $q$ not a prime power seems more complicated, especially since nontrivial $1$-perfect codes are still not known in this case.

The crucial idea of the proof of Theorem \ref{T:prime} is to find analytically some elements of the quotient matrix of the distance partition of initial code in $i$-distance Hamming graph for some $i$ and then analyze it. New necessary conditions from Theorems \ref{T:general_q_2} and \ref{T:general_J} require further analysis and may exclude some infinite series of parameters. 

\section{Acknowledgments} 
The author was supported by the NSP P. Beron project CP-MACT. Author is grateful to Fedor Petrov for useful discussions on the proof of Theorem \ref{T:general_q}.







\bibliography{Extended_perfect_codes}

\end{document}